\documentclass{amsart}
\usepackage{setspace}
\usepackage{verbatim}

\usepackage{amsmath}
\usepackage{amsthm, dsfont}
\usepackage{amssymb}
\usepackage{enumerate}
\usepackage[latin1]{inputenc}
\usepackage{xcolor}
\newtheorem{theorem}{Theorem}[section]
\newtheorem{lemma}[theorem]{Lemma}
\newtheorem{proposition}[theorem]{Proposition}
\newtheorem{corollary}[theorem]{Corollary}
\newtheorem{remark}[theorem]{Remark}

\numberwithin{equation}{section}

\newcommand{\Q}{{\mathbb{Q}}}

\newcommand{\N}{\mathds{N}}

\newcommand{\OL}{\mathcal{O}}

\begin{document}
\title[On quaternion algebras over the dihedral number  fields ...]
{On quaternion algebras over  some extensions of quadratic  number fields}

\author{Vincenzo Acciaro}
\address{Vincenzo Acciaro, Dipartimento di Economia, Universit\`a di Chieti--Pescara,
Viale della Pineta, 4, 65127 Pescara, Italy}
\email{v.acciaro@unich.it}

\author{Diana Savin}
\address{Diana Savin, Faculty of Mathematics and Computer Science, Ovidius University,
Bd. Mamaia 124, 900527, Constanta, Romania}
\email{savin.diana@univ-ovidius.ro}

\author{Mohammed Taous}
\address{Mohammed Taous, Department of  Mathematics, Faculty of  Sciences and Technology, Moulay Ismail University, Errachidia, Morocco.}
\email{taousm@hotmail.com}

\author{Abdelkader Zekhnini}
\address{Abdelkader Zekhnini, Mohammed First University, Pluridisciplinary Faculty of Nador, Department Mathematics and Informatics, Nador, Morocco}
\email{zekha1@yahoo.fr}

\date{}
\begin{abstract}
Let $p$ and $q$ be two positive primes. Let $\ell$ be an odd positive prime integer and $F$  a quadratic number field. Let $K$ be an extension of $F$ such that $K$ is a dihedral extension of $\Q$ of degree $\ell$ over $F$  or $K$ is an abelian $\ell$-extension unramified over $F$ assuming $\ell$ divides the class number of  $F$. In this paper, we obtain a complete characterization of
division quaternion algebras $H_{K}(p, q)$ over   $K$.
\end{abstract}
\maketitle
\section{Introduction}
\noindent
Let $F$ be a field with $char(F)\neq2$ and let $a, b \in F \backslash \{0\}$.
{The   generalized quaternion algebra $H_{F}(a, b)$
 is the associative algebra generated over the field $F$ by two elements $i$ and $j$,  subject to the relations
$i^2 = a$, $j^2 = b$ and $ij = -ji$.

Quaternion algebras turn out to be  central simple algebras  of dimension $4$ over $F$. A basis for $H_{F}(a, b)$
over $F$  is given by $\{1,i,j,ij\}$.

It can be shown that every four dimensional central simple algebra over a field of characteristic $\neq 2$ is a quaternion algebra.

If  $x=x_{1}  1+ x_{2}i+x_{3}j+x_{4}k \in H_{F}(a, b),$
with $x_{i}\in F$,       the conjugate $\overline{x}$ of $x$ is defined as
$\overline{x}=x_{1}   1  -x_{2}i -x_{3}j-x_{4}k$, and the norm of $x$ as $\boldsymbol{n}(
x) =x  \overline{x}=x_{1}^{2}-a x_{2}^{2}-b x_{3}^{2}+ ab x_{4}^{2}.$


If the equations $ax=b,\ ya=b$   have unique solutions    for all  $a, b \in A, \ a\neq 0$,
then the algebra $A$ is called  {\em a division algebra}.  If $A$ is a
finite-dimensional algebra, then $A$ is a division algebra if and only if $A$
has no zero divisors.
In the case of  generalized quaternion algebras  there is a simple criterion that guarantees them to be division algebras:   $H_{F}  (a,b)$ is a division algebra if and only if there is a unique element of zero norm, namely $x = 0$.

Let $L$ be an extension field of $F$, and let $A$ be a central simple algebra over  $F$.
We recall that $A$ is said to   {split} over $L$,
and $L$ is called
a   {splitting field} for $A$,
if  $A\otimes _{F} L$ is isomorphic to a  full
matrix algebra over $L$.

Several results are known about   the splitting behavior of quaternion algebras and symbol algebras over specific fields
\cite{chapman 2016, chapman 2017, gille, lam, Rowen}.

Explicit conditions which guarantee that a quaternion algebra  splits over the field of rationals numbers,
or else is a division algebra, were studied in  \cite{alsina}.

In   \cite{astz} we studied the splitting behavior of some quaternion algebras over quadratic  fields.
Then, in   \cite{astz2}  we extended the previous results to
the composite of quadratic number fields.

Traditionally, in order to decide whether a quaternion algebra is a division algebra or it splits,
 one either looks  for the primes which ramify in the algebra or otherwise one has to appeal to  Hasse's norm theorem which allows
to reduce the problem to   computations of  local Hilbert symbols  \cite{Serre,vostokov}.
Very recently,  Goldstein  \cite{goldstein}  tried to combine these two techniques.

In this paper, as we did in our previous papers \cite{astz,astz2},
we adopt the former approach, i.e we study the ramification
 of   certain integral primes,  and we obtain a nice characterization of quaternion division algebras
 $H_{K}(p, q)$  solely in terms of quadratic residues, assuming that $p$ and $q$ are  positive primes
and $K$ is  a  dihedral extension of $\Q$ of prime degree over an imaginary quadratic field. The layout of the paper is the following.
In Section \ref{pre}  we state some preliminary results    which we will need   later.
In Section \ref{dihedral} we apply these results  to study quaternion algebras over
dihedral extensions $K$ of $\Q$ of
prime degree $l$ over an imaginary quadratic field.

\section{  Preliminary results }
\label{pre}
In this section we recall some basic   results concerning quaternion algebras.
Unless otherwise stated, when we say ``prime integer" we mean ``positive prime integer".

Let $K$ be a number field  and let $\mathcal{O}_{K}$ be its ring of integers.
If $v$ is a place of $K$, let us denote by $K_v$ the completion of $K$ at $v$.
We recall that a quaternion algebra $H_{K}(a,b)$ is said to ramify
 at a place $v$ of $K$ - or $v$ is said to ramify in $H_{K}(a,b)$ -  if the  quaternion $K_{v}$-algebra
 $H_{v}=K_{v}\otimes H_{K} (a,b )$
is a division algebra.
This happens  exactly when the Hilbert symbol
 $(a, b)_{v}$ is equal to $-1$, i.e. when the equation $ax^2+by^2=1$ has no solutions in $K_v$.
We also recall that the reduced discriminant $D_{H_{K}(a,b)}$ of the quaternion algebra  $H_{K}(a,b)$
is defined as the product of those prime ideals of the ring of integers $\mathcal{O}_{K}$ of $K$
which  ramify in $H_{K} (a,b )$.
The following splitting criterion for a quaternion algebras is well  known \cite[Corollary 1.10]{alsina}:
\begin{proposition}
\label{twodotfour}
 Let $K$   be a number field. Then, the quaternion algebra ${H}_{K}(a,b) $  is split if and only if
 its discriminant
  $D_{H_{K}(a,b)}$ is equal to  the ring of integers $ \mathcal{O}_{K}$ of $K$.
\end{proposition}
If $\mathcal{O}_{K}$ is a principal ideal domain, then we may identify the ideals
of $\mathcal{O}_{K}$ with their generators, up to units.
Thus, in a quaternion algebra $H$ over $\mathbb{Q}$, the element $D_{H}$ turns out to be an integer, and $H$  is split if and only if $D_{H}=1$.

The next proposition gives us a geometric interpretation of splitting \cite[Proposition 1.3.2]{gille}:
\begin{proposition}
\label{twodotfive}
Let  $K$   be a field. Then, the quaternion algebra  ${H}_{K} (a,b ) $ is split if
and only if the
conic   $C(a, b) :$ $ax^{2}+by^{2}=z^{2}$
 has a rational point in $K$,  i.e.   there are $x_{0},y_{0},z_{0}\in K$
 such that $ax_{0}^{2}+by_{0}^{2}=z_{0}^{2}$.
\end{proposition}
The next proposition due to Hasse relates the norm group of extensions of the base
 field to the splitting behavior of a quaternion  algebra  \cite[ Proposition 1.1.7]{gille}:
\begin{proposition}
\label{twodotsix}
Let   $F$   be a field. Then, the quaternion algebra
${H}_{F} (a,b ) $    is split if and only if   $a$ is the norm of an element of $F (\sqrt{b} )$.
 \end{proposition}
For quaternion algebras it is true the following \cite[Proposition 1.1.7]{gille}:
\begin{proposition}
\label{twodoteightone}
Let   $K$   be a field with   char $K \neq 2$   and let  $a, b \in K \backslash \{  0 \}$.
 Then the quaternion algebra $H_{K} ( a, b )$   is either split or a division algebra.
\end{proposition}
In particular, this tells us that a quaternion algebra $H_\Q (a, b )$  is a division algebra if and only if there is   a prime
$p$ such that $p | D_{H_\Q  (a, b )}$.

It is known   \cite{kohel} that  if a prime integer $p$ divides  $D_{H(a, b)}$ then it must divide  $2ab$,
hence we may restrict our attention to these primes.
In other words, in order to obtain a sufficient condition for a quaternion algebra
$H_{\mathbb{Q}(\sqrt{d})}(p,q)$   to be a division algebra,
 it is important to study the ramification of the primes   $ 2,p,q $     in the   algebra    $H_\Q(p, q)$.
The following lemma from the classical book by Alsina  \cite[Lemma 1.21]{alsina} gives us a hint:
\begin{lemma}
\label{threedotfour}
Let  $p$ and $q$ be two primes, and let $H_{\mathbb{Q}}(p,q)$ be a quaternion algebra of discriminant $D_H$.
\begin{enumerate}[\rm i.]
\item
if $p\equiv q\equiv 3 \pmod4$   {and} $(\frac{q}{p})\neq 1$,   {then} $D_{H}=2p$;
 \item
if $q=2$ and  $p\equiv 3 \pmod8$,  {then} $D_{H}=pq=2p$;
\item
if  $p$   {or} $q\equiv1 \pmod4$ , with  $p\neq q$  {and} $(\frac{p}{q})=-1$,   then  $D_{H}=pq$.
\end{enumerate}
\end{lemma}

We recall that a small ramified $\mathbb{Q}$-algebra is a rational
quaternion algebra having the discriminant equal to the product of two distinct prime numbers.
The following necessary and sufficient explicit condition for a small ramified $\mathbb{Q}$-algebra $H_\Q(p,q)$ to be a division algebra over a quadratic field $\mathbb{Q}(\sqrt{d})$ was proved in \cite {astz}:
\begin{proposition}
\label{threedotsix}
  {Let} $p$ and $q$   {be two distinct  odd primes,}  with  $p$   {or} $q \equiv 1\pmod4$ {and} $(\frac{p}{q})=-1.$    Let  $K=\mathbb{Q}(\sqrt{d})$   {and let} $\Delta_{K}$   {be the discriminant of} $K.$
  {Then the quaternion algebra} $H_{\mathbb{Q}(\sqrt{d})}(p,q)$   {is a division algebra if and only if} $(\frac{\Delta_{K}}{p})= 1$   {or} $(\frac{\Delta_{K}}{q})=1$.
\end{proposition}

When  $q=2$ and $p$ is a prime such that $p \equiv 3$ (mod 8),  then,
  according to Lemma \ref{threedotfour} the discriminant
$D_{H_{\mathbb{Q}(p,q)} }$ is equal to $2p$,
so $H_{\mathbb{Q}}  (p,q  )$ is a division algebra.
The next proposition, which was proved in \cite{astz},  shows what
 happens when we extend the field of scalars from $\Q$ to
$\Q(  \sqrt{d} )$:
\begin{proposition}
\label{threedotseven}
 {Let} $p$   {be an odd prime,} with $p \equiv 3\pmod8$.   Let  $K=\mathbb{Q}(\sqrt{d})$   {and let} $\Delta_{K}$   {be the discriminant of} $K.$
  Then  $H_{\mathbb{Q}(\sqrt{d})}(p,2)$   {is a division algebra if and only if} $(\frac{\Delta_{K}}{p})= 1$   {or} $d$ $\equiv 1\pmod8$.
\end{proposition}

When $p$ and $q$ are primes both congruent to $3$ modulo $4$
and   $(\frac{q}{p})\neq 1$ then,
according to Lemma \ref{threedotfour}(i)   the discriminant
$D_{H_{\mathbb{Q}}   (p,q)}$ is equal to $2p$, so $H_{\mathbb{Q}}(p,q)$ is a division algebra.
The next proposition which also was proved in \cite{astz},   tells us when
the quaternion algebra $H_{\mathbb{Q}(\sqrt{d})}(p,q)$ is still a division algebra:
\begin{proposition}
\label{threedoteight}
 {Let} $p$ and $q$   {be two odd prime integers,} with $p\equiv q\equiv 3\pmod4$  and  $(\frac{q}{p})\neq 1$.
 Let  $K=\mathbb{Q}(\sqrt{d})$   {and let} $\Delta_{K}$   {be the discriminant of} $K.$
  {Then the quaternion algebra} $H_{\mathbb{Q}(\sqrt{d})}(p,q)$   {is a division algebra if and only if} $(\frac{\Delta_{K}}{p})= 1$   {or} $d$ $\equiv 1\pmod8$.
\end{proposition}

\section{  Main results}
\label{dihedral}
Let's ask ourselves now what happens when  we consider a quaternion algebra over  a Galois extension $K$ of $\mathbb{Q}$,
with nonabelian Galois group  of  order $2l$, where $l$ is an odd prime integer.
For this purpose, we recall the following result, which can be found as an exercise in \cite[p. 77]{lam}:
\begin{remark}
 \label{sixdotone} Let $K/F$ be a finite extension of fields of odd degree, and let $a, b \in F \backslash \{0\}$.
Then the quaternion algebra $H_{K}(a,b)$ splits if and only if $H_{F}(a,b)$ splits.
\end{remark}
Let us first consider the case  $Gal( K/ \mathbb{Q}) \cong S_{3}$, i.e. the dihedral group $D_{3}$.
The following three propositions will deal with this case:
\begin{proposition}
 \label{sixdottwo}
Let $\epsilon$  be a primitive  third root of unity, and put $F =\mathbb{Q}\left(\epsilon\right)$.
  Let $\alpha \in K\backslash \{0, 1\}$ be a cubicfree integer,  put $K=F\left(\sqrt[3]{\alpha}\right)$
  and let  $p$, $q$  be two distinct odd prime integers such that
   $(\frac{p}{q})=-1$ and $p$  {or} $q\equiv1\pmod4$. {Then the quaternion algebra} $H_K(p,q)$   {is a division algebra if and only if} $(\frac{-3}{p})=1$ {or} $(\frac{-3}{q})=1.$
\end{proposition}
\begin{proof}
Clearly $F =\mathbb{Q}\left(\epsilon\right)=\mathbb{Q}\left(i\sqrt{3}\right)$ is an imaginary  quadratic number field
and $\left[K:F\right]=3$. Moreover,
$K/\mathbb{Q}$ is Galois  and Gal$( K/\mathbb{Q}) \cong S_{3}$.
According to Remark \ref{sixdotone} and Proposition \ref{twodoteightone}, $H_K(p,q)$ is a division algebra if and only if $H_F(p,q)$ is a division algebra.
By Proposition \ref{threedotsix}, this can happen if and only if $(\frac{-3}{p})=1$ {or} $(\frac{-3}{q})=1.$
\end{proof}
\begin{proposition}
 \label{sixdottthree}
Let $\epsilon$ be a primitive  third root of unity, and put $F =\mathbb{Q}\left(\epsilon\right)$.
  Let also $\alpha \in K \backslash \{0, 1\}$ be a cubicfree integer, put  $K=F\left(\sqrt[3]{\alpha}\right)$
 and {let} $p$ be {an odd prime integer} such that $p \equiv 3\pmod8$. {Then the quaternion algebra} $H_K(p,2)$   {is a division algebra if and only if} $(\frac{-3}{p})=1.$
\end{proposition}
\begin{proof}
The proof is obtained  from the proof of Proposition \ref{sixdottwo} by replacing
Proposition \ref{threedotsix} with Proposition \ref{threedotseven}.
\end{proof}
\begin{proposition}
 \label{sixdotfour}
Let $\epsilon$ be a primitive  third root of unity, and put  $F =\mathbb{Q}\left(\epsilon\right)$.
  Let $\alpha \in K\backslash \{0, 1\}$ be a cubicfree integer, put  $K=F\left(\sqrt[3]{\alpha}\right)$
and {let} $p$ and $q$  {be distinct odd prime integers} with
$(\frac{q}{p})\neq 1$ and $p$$\equiv$$q$$\equiv3$ (mod $4$). {Then the quaternion algebra} $H_K(p,q)$   {is a division algebra if and only if} $(\frac{-3}{p})=1.$
\end{proposition}
\begin{proof}
The proof is obtained from the proof of Proposition \ref{sixdottwo} after replacing
Proposition \ref{threedotsix} with Proposition \ref{threedoteight}.
\end{proof}

In what follows,  let $\ell$ be an odd positive prime integer and $F$  a quadratic number field. Let $K$ be an extension of $F$ defined as follows: $K$ is a dihedral extension of $\Q$ of prime degree $\ell$ over $F$,  or $K$ is an abelian $\ell$-extension unramified over $F$ assuming $\ell$ divides the class number of  $F$ (so $[K: F]=\ell^n$ which is odd, with $n\in\N^*$).
The existence of such a $K$ containing $F$ is guaranteed, in the second case, by class field theory, but  the first one is a tipical problem in inverse Galois theory.
The following result from \cite[p. 352-353]{jensen} guarantees us that such a $K$ indeed exists (for the first case):
\begin{theorem}
 \label{sixdotfive}
For any prime $\ell$ and any quadratic field $F =\mathbb{Q}(\sqrt{d})$
there exist infinitely many dihedral fields $K$ of degree $2\ell$ containing $F$.
\end{theorem}
 If the quaternion algebra $H_{F}(p,q)$ is a division algebra,
 we would like to know when $H_{K}(p,q)$   is still a division algebra.
The following three propositions will allow us
to achieve this task:
\begin{proposition}
\label{sixdotsix}
Let  $F$ be a  quadratic field  and $\Delta_{F}$ its discriminant. Let $K$ be an extension of $F$  defined as above. Let $p$ and $q$  be distinct odd prime integers, with
$(\frac{p}{q})=-1$, and $p$  {or} $q\equiv1\pmod4$. Then the quaternion algebra $H_{K}(p, q)$   {is a division algebra if and only if} $\left(\frac{\Delta_{F}}{p}\right)=1$ {or} $\left(\frac{\Delta_{F}}{q}\right)=1.$
\end{proposition}
\begin{proof}
Note first  that the degree $\left[K: F\right] $ is odd, since it is equal to $\ell$ or $\ell^n$ with $n\in\N^*$.
According to Remark \ref{sixdotone} and Proposition \ref{twodoteightone}, $H_{K}(p,q)$ is a division algebra if and only if $H_{F}(p,q)$ is a division algebra. By Proposition \ref{threedotsix}, this can happen if and only if $\left(\frac{\Delta_{F}}{p}\right)=1$ {or} $\left(\frac{\Delta_{F}}{q}\right)=1.$
\end{proof}
\begin{proposition}
\label{sixdotseven}
{Let} $d\neq1$  {be a squarefree integer} and let $F =\mathbb{Q}(\sqrt{d})$ be a  quadratic number field  and  $\Delta_{F}$ its discriminant. Let $K$ be an extension of $F$ defined  as above. {Let} $p\equiv3 \pmod8$ be an odd prime integer. {Then the quaternion algebra} $H_{K}(p,q)$   {is a division algebra if and only if} $\left(\frac{\Delta_{F}}{p}\right)=1$ {or} $d\equiv1\pmod8$.
\end{proposition}
\begin{proof}
The proof is obtained from the proof of Proposition \ref{sixdotsix}, after replacing
Proposition \ref{threedotsix} with Proposition \ref{threedotseven}.
\end{proof}
\begin{proposition}
\label{sixdoteight}
{Let} $d\neq1$  {be a squarefree integer} and let $F =\mathbb{Q}(\sqrt{d})$ be a  quadratic number field  and  $\Delta_{F}$ its discriminant. Let $K$ be an extension of $F$ defined  as above. {Let} $p$ and $q$  {be distinct odd prime integers}
with $(\frac{q}{p})\neq 1$ and $p\equiv q\equiv3\pmod4$. {Then the quaternion algebra} $H_{K}(p, q)$   {is a division algebra if and only if} $\left(\frac{\Delta_{F}}{p}\right)=1$ {or} $d\equiv1\pmod8$.
\end{proposition}
\begin{proof}
The proof is obtained from the proof of Proposition \ref{sixdotsix} after replacing
Proposition \ref{threedotsix} with Proposition \ref{threedoteight}.
\end{proof}
We conclude our discussion
with the main theorem of the paper:
\begin{theorem}
\label{teorema}
{Let} $d\neq1$  {be a squarefree integer} and $\ell$ an odd positive prime integer. Let $F =\mathbb{Q}(\sqrt{d})$ be a  quadratic number field and $\Delta_{F}$ its discriminant. Let $K$  be a dihedral extension of $\Q$ of prime degree $\ell$ over $F$  or $K$ is an abelian $\ell$-extension unramified over $F$ whenever $\ell$ divides the class number of $F$.
 Let $p$ and $q$ be  two distinct odd prime integers.
   Then the quaternion algebra $H_{K}(p,q)$   {is a division algebra if and only if}
   one of the following conditions is verified:
\begin{enumerate}[\rm1.]
  \item
  $p$ or  $q \equiv 1  \pmod4$, $(\frac{p}{q})=-1$,
  and  $\left(\frac{\Delta_{F}}{p}\right)=1$ {or} $\left(\frac{\Delta_{F}}{q}\right)=1$;
  \item
$p \equiv 3 \pmod8$,
and $\left(\frac{\Delta_{F}}{p}\right)=1$ {or} $d\equiv1\pmod8$;
  \item
    $p \equiv q \equiv 3\pmod4$,
 $(\frac{q}{p})\neq 1$,
and
 $\left(\frac{\Delta_{F}}{p}\right)=1$ {or} $d\equiv1\pmod8$.
 \end{enumerate}
 \end{theorem}


\begin{thebibliography}{99}

\bibitem{astz}
V. Acciaro, D. Savin, M. Taous and A. Zekhnini,
On quaternion algebras over quadratic number fields,
preprint,
2019.

\bibitem{astz2}
V. Acciaro, D. Savin, M. Taous and A. Zekhnini,
On quaternion algebras over the composite of quadratic number fields,
preprint,
2019.


\bibitem{alsina}
M. Alsina and P. Bayer, Quaternion Orders, Quadratic Forms and Shimura Curves, CRM Monograph Series,
22, American Mathematical Society, 2004.


\bibitem{chapman 2016} A. Chapman, D. J.Grynkiewiczb, E. Matzri, L. H.Rowen, U. Vishne, Kummer spaces in symbol algebras of prime degree, J. of Pure and Applied Algebra, vol. 220, issue 10 (2016): 3363-3371.

\bibitem{chapman 2017} A. Chapman, Symbol length of p-algebras of prime exponent, J. of Algebra and its Applications, vol. 16, No. 7
(2017), 1750136 (9 p.).

\bibitem{chinburg}
T. Chinburg and E. Friedman, An embedding theorem for quaternion algebras, J. London Math. Soc. (2), 60(1):33--44, 1999.



\bibitem{gille}
P. Gille, T. Szamuely,  Central Simple Algebras and Galois
Cohomology, Cambridge University Press, 2006.

 \bibitem{goldstein}  D. Goldstein and M. Schacher, Norms in central simple algebras, Pacific Journal of Mathematics, Vol. 292, No. 2, 2018, pp. 373-388.


\bibitem{jensen} C.U. Jensen, N. Yui, Polynomials with $D_{p}$ as Galois Group, Journal of  Number Theory  {15}, (1982)
pp. 347-375.

\bibitem{kohel}
D.R. Kohel, {Quaternion algebras,}{ available online at http://www.i2m.univ-amu.fr/perso/david.kohel/alg/doc/AlgQuat.pdf.}

\bibitem{lam}
T. Y. Lam,  Introduction to Quadratic Forms over Fields,
American Mathematical Society, 2004.




\bibitem{Rowen} L. Rowen, D. J. Saltman, Tensor products of division algebras and fields, Journal of Algebra, {394} (2013), pp. 296-309.




\bibitem{Serre} J. P. Serre, Local fields, Graduate Texts in Mathematics 67, Springer, New York, 1979.

\bibitem{vostokov}
S.V. Vostokov, Explicit formulas for the Hilbert symbol, Geometry and Topology Monographs, Vol.3: Invitation to higher local fields,
Part I, section 8, pp 81-89.


\end{thebibliography}
\end{document}